\documentclass[11pt,twoside,a4]{article}

\usepackage{fancyhdr}
\usepackage{amsmath}
\usepackage{amssymb}
\usepackage{amsfonts}
\usepackage{hyperref}
\usepackage{url}

\usepackage{pifont}

\renewcommand{\thefootnote}{\fnsymbol{footnote}}

\long\def\sfootnote[#1]#2{\begingroup
\def\thefootnote{\fnsymbol{footnote}}\footnote[#1]{#2}\endgroup}

\newtheorem{theorem}{Theorem}

\newtheorem{lemma}[theorem]{Lemma}
\newtheorem{corollary}[theorem]{Corollary}

\newenvironment{proof}{\noindent\mbox{\bf Proof.}}
{\hfill\mbox{\ding{111}}\bigskip}

\begin{document}
\pagestyle{fancy}
\lhead[{\sl Axiomatizing Mathematical Theories:  Multiplication}]{\thepage}
\chead[]{}
\rhead[\thepage]{\sl Saeed Salehi}
\lfoot[]{}
\cfoot[]{}
\rfoot[]{}
\renewcommand{\headrulewidth}{0pt}
\renewcommand{\footrulewidth}{0pt}
\thispagestyle{empty}

\begin{center}
{\Large Axiomatizing Mathematical Theories:   Multiplication}\footnote{Published in:  {\sc A. Kamali-Nejad} (ed.) {\em 
Proceedings of Frontiers in Mathematical Sciences}, Sharif University of Technology, 25--27 December 2012, Tehran, 
Fundamental Education Publications (Iran 2012) pages 165--176.}

\bigskip

{\sc Saeed Salehi}\footnote{The author is partially supported by grant  N$^{\underline{\rm o}}$ 91030033 of the Institute for Research in Fundamental Sciences ($\bigcirc\!\!\!\!/\!\!\!\bullet\!\!\!/\!\!\!\!\bigcirc$\ $\mathbb{I}\mathbb{P}\mathbb{M}$), Niavaran, Tehran, Iran.}

Department of Mathematical Sciences, University of Tabriz, 29 Bahman Boulevard,

P.O.Box 51666--17766, Tabriz, Iran.

School of Mathematics,
              Institute for Research in Fundamental Sciences (IPM),

P.O.Box 19395--5746,  Niavaran, Tehran, Iran.

{\tt http://saeedsalehi.ir/ \qquad  salehipour}@{\tt tabrizu.ac.ir}

\end{center}

\begin{abstract}
Axiomatizing mathematical structures is a goal of Mathematical Logic.
Axiomatizability of the theories of some structures have turned out to be
quite difficult and challenging,
and some remain open. However  axiomatization of some mathematical structures
are now classical theorems in Logic, Algebra and Geometry. In this paper we will
study the axiomatizability of the theories of  multiplication  in the domains of natural,
integer, rational, real, and complex numbers. We will review some classical theorems,
and will give some new proofs for old results. We will see that some structures are
missing in the literature, thus leaving it open whether the theories of that
structures are axiomatizable (decidable) or not. We will answer one of those open
 questions in this paper.



\end{abstract}

\section{Introduction \hfill\fbox{\small {\sl Dedicated to Professor {\sc Siavash Shahshahani} for his 70th Birthday.}}}\label{intro}
Classic Algebraic Geometry is about studying properties of the sets
of zeros of polynomials; the word geometry refers to the set of
zeros of (a system of) polynomials. Studying the systems of linear
equations with one variable is a subject of Linea Algebra, and
studying polynomial equations in one variable is a subject of
Algebra. In Algebraic Geometry these two areas are combined by
studying systems of polynomials in several variables.
Chevalley--Tarski Theorem in algebraic geometry states that the
projection of a constructible set is constructible;  in mathematical
logic this theorem  implies that the theory of  algebraically closed
fields, or equivalently the theory of the structure
$\langle\mathbb{C},0,1,+,-,\cdot\rangle$, admits quantifier
elimination. As a result the theory of the structure
$\langle\mathbb{C},+,\cdot\rangle$ is decidable and can be
axiomatized as an \textsl{algebraically closed field}. A set $D$ is called
decidable when there exists an algorithm which for a given input $x$
outputs {\tt yes} if $x\in D$ and outputs {\tt no} if $x\not\in D$.
Tarski--Seidenberg Theorem (or Tarski--Seidenberg Principle) in
algebraic geometry states that the projection of a  semialgebraic
set is a semialgebraic set;  in mathematical logic this theorem
implies that the theory of the real closed (ordered) fields, or
equivalently the theory of the structure
$\langle\mathbb{R},0,1,+,-,\cdot,<\rangle$, admits quantifier
elimination. As a result  the theory of the structure
$\langle\mathbb{R},+,\cdot\rangle$ is decidable and can be
axiomatized as a \textsl{real closed (ordered) field}.   Let us note that
ordering can be defined by addition and multiplication in
$\mathbb{R}$:  for real $a,b$ we have $a\!\leqslant\!b$ if and only
if $\exists x (a\!+\!x^2\!=\!b)$.

One of the most surprising and most significant results in the
twentieth century mathematics was G\"odel's Incompleteness Theorem.
A semantical reading of this theorem is that the the additive and
multiplicative  theory of  the natural numbers,
$\langle\mathbb{N},+,\cdot\rangle$, is not decidable. Undecidability
of the theory of the structure $\langle\mathbb{N},+,\cdot\rangle$
means that there is no algorithm which for a given first order
sentence $\varphi$ in the language $\{+,\cdot\}$ (as input) can
decide (output {\tt yes}) if $\mathbb{N}\models\varphi$ or not
(output {\tt no} if $\mathbb{N}\not\models\varphi$). The theory of
 $\langle\mathbb{Z},+,\cdot\rangle$ is not decidable
either because $\mathbb{N}$ is definable in it, and also the theory
of  $\langle\mathbb{Q},+,\cdot\rangle$ is undecidable
since $\mathbb{Z}$ is definable in it.

One surprising aspect of G\"odel's 1931 incompleteness theorem
(undecidability of the theory $\langle\mathbb{N},+,\cdot\rangle$)
was that by the results of Presburger and Skolem the additive theory
of the natural numbers without multiplication, $\langle\mathbb{N},+\rangle$,
and the multiplicative theory of the natural numbers without addition,
$\langle\mathbb{N},\cdot\rangle$, were proved (or announced) to be decidable by 1930.
So, the theory of the structures $\langle\mathbb{N},+,\cdot\rangle$ was expected to be
decidable, which turned out not to be so by Kurt G\"odel. Indeed the theories of the
structures $\langle\mathbb{Z},+\rangle$ and $\langle\mathbb{Z},\cdot\rangle$ are
also decidable, and the theories of the structures
$\langle\mathbb{R},+\rangle$, $\langle\mathbb{R},\cdot\rangle$,
$\langle\mathbb{C},+\rangle$ and $\langle\mathbb{C},\cdot\rangle$ are
evidently decidable by the above mentioned theorems of Tarski
(and Chevalley and Seidenberg). Our aim is to study the following structures of
 multiplication  in the domains of $\mathbb{N},\,\mathbb{Z},\,\mathbb{Q},\,\mathbb{R}$ and $\mathbb{C}$.\\[-1.5em]
\begin{center}
\begin{tabular}{|c||c|c|c|c|c|}
\hline
         & $\mathbb{N}$ & $\mathbb{Z}$ & $\mathbb{Q}$ & $\mathbb{R}$ &
         $\mathbb{C}$ \\
\hline
 \hline
        $\{\cdot\}$ & $\langle\mathbb{N},\cdot\rangle$ & $\langle\mathbb{Z},\cdot\rangle$
        & $\langle\mathbb{Q},\cdot\rangle$ & $\langle\mathbb{R},\cdot\rangle$ &
        $\langle\mathbb{C},\cdot\rangle$ \\
  \hline
                $\{+,\cdot\}$ & $\langle\mathbb{N},+,\cdot\rangle$ &
                 $\langle\mathbb{Z},+,\cdot\rangle$
        & $\langle\mathbb{Q},+,\cdot\rangle$ & $\langle\mathbb{R},+,\cdot\rangle$ &
        $\langle\mathbb{C},+,\cdot\rangle$ \\
  \hline
            \end{tabular}
\end{center}
Though the theory of the structures $\langle\mathbb{R},\cdot\rangle$
and $\langle\mathbb{C},\cdot\rangle$ are  decidable by Tarski's
theorem (decidability of the theories of
$\langle\mathbb{R},+,\cdot\rangle$ and
$\langle\mathbb{C},+,\cdot\rangle$), we will prove this fact
directly (without appealing to Tarski's results) by quantifier
elimination. Here, the decidability or undecidability of the
structure $\langle\mathbb{Q},\cdot\rangle$ seems to be missing in
the literature. The theory $\langle\mathbb{Q},+,\cdot\rangle$ is not
decidable, so one cannot immediately infer the decidability of the
theory of $\langle\mathbb{Q},\cdot\rangle$. In this paper we show
the decidability of the theory of this structure by the method of
quantifier elimination. Whence, we will give some nice
characterizations for the theory of multiplication in the fields of
rational, real, and complex numbers. The status of decidability of
the theories of the above mentioned structures has been summarized
 in the below table, where  decidable theories are indicated by $\Delta_1$
 and undecidable theories by $\Delta_1\!\!\!\!\!\!/\!\!\!\!\backslash$   \ ).\\[-1.5em]
\begin{center}
\qquad \begin{tabular}{|c||c|c|c|c|c|} \hline
         & $\mathbb{N}$ & $\mathbb{Z}$ & $\mathbb{Q}$ & $\mathbb{R}$ &
         $\mathbb{C}$ \\
\hline
  \hline
          $\{\cdot\}$ & $\Delta_1$ & $\Delta_1$
        & $\Delta_1$ & $\Delta_1$ &
        $\Delta_1$ \\
  \hline
          $\{+,\cdot\}$ & $\!\!\!\!\!\Delta_1\!\!\!\!\!\!/\!\!\!\!\backslash$
          & $\!\!\!\!\!\Delta_1\!\!\!\!\!\!/\!\!\!\!\backslash$
        & $\!\!\!\!\!\Delta_1\!\!\!\!\!\!/\!\!\!\!\backslash$ & $\Delta_1$ &
        $\Delta_1$ \\
  \hline
            \end{tabular}
\end{center}

\section{The Theory of Multiplication}

\subsection{Natural and Integer Numbers}
There are some model--theoretic (and advanced) proofs for the
decidability of the theory of $\langle\mathbb{N},\cdot\rangle$ in
the literature, the most elementary (and in the author's opinion,
the most interesting) proof is the one given by Patrick Cegielski in
\cite{cegielski} which uses the technique of quantifier elimination.
Unfortunately, there is no English version of this proof (see
\cite{smorynski} for an overview of it). As far as the author knows,
there is no direct reference to the decidability of the theory of
the structure $\langle\mathbb{Z},\cdot\rangle$ in the literature,
but somehow it seems that Cegielski's quantifier elimination proof
can also be adopted to the case of $\mathbb{Z}$. Thus the
multiplicative theory of $\mathbb{Z}$ is most likely decidable, and
most likely can be proved by the method of quantifier elimination.

\subsection{Rational Numbers}

Let us note that the language $\{\cdot\}$ does not
 allow quantifier elimination for $\langle\mathbb{Q},\cdot\rangle$, since e.g.
 the formula $\exists y[x=y^2]$ is not equivalent to a quantifier--free formula. So,
 we first restrict our attention to the set of positive rational
  numbers $\mathbb{Q}^+=\{r\!\in\!\mathbb{Q}\mid r\!>\!0\}$ and extend the
 language to $\mathcal{L}=\langle 1,\cdot, ^{-1},R_2,R_3,\ldots\rangle$, where $R_n$
 is interpreted as ``being the $n$th power of a rational"; or in the other words
  $R_n(x)\equiv \exists y[x=y^n]$.
We note that the quantifier--free formulas of $\mathcal{L}$ are
decidable: for any given rational number $r$ and any natural $n$ one
can decide if $r$ is an $n$th power of (an-)other rational number or
not. Thus, quantifier elimination in
$\langle\mathbb{Q}^+,\mathcal{L}\rangle$ implies the decidability of
the theory of the structure
$\langle\mathbb{Q}^+,\mathcal{L}\rangle$, and hence decidability of
the theory of  $\langle\mathbb{Q}^+,\cdot\rangle$. We will need the
following general form of the Chinese Remainder Theorem
(\cite{ore}).
\begin{theorem}[General Chinese Remainder]\label{th-crt}
The system $\{x\equiv_{n_i}u_i\}_{i=1}^{l}$ of congruence equations
has a solution in $\mathbb{Z}$ if and only if for every $i\not=j$,
$u_i\equiv_{(n_i,n_j)}u_j$ where $(n_i,n_j)$ is the greatest common
divisor of $n_i$ and $n_j$. Moreover the solution $x_0$ {\rm (}if
exists{\rm )} is unique module $n=[n_1,\ldots,n_l]$ {\rm (}the least
common multiple of $n_i$'s{\rm )}, and so all the solutions will be
of the form $N\!\cdot\!n+x_0$ for arbitrary $N\in\mathbb{Z}$.
\end{theorem}
\begin{proof} If $x_0$
satisfies $\bigwedge_i x_0\equiv_{n_i}u_i$ then for any $i\not=j$ we will
have $x_0\equiv_{(n_i,n_j)}u_i$ and $x_0\equiv_{(n_i,n_j)}u_j$, whence
$u_i\equiv_{(n_i,n_j)}u_j$.
Conversely, if $\bigwedge_{i\not=j}u_i\equiv_{(n_i,n_j)}u_j$ then we
show the existence of some $x_0$ which satisfies $\bigwedge_i x_0\equiv_{n_i}u_i$.
Then of course every number $x=N\!\cdot\!n+x_0$ satisfies the system of equations
 $\bigwedge_i x\equiv_{n_i}u_i$  as well, and every solution $y$ of the equations
  $\bigwedge_i y\equiv_{n_i}u_i$ satisfies $\bigwedge_i x_0\equiv_{n_i}y$, and so
  $x_0\equiv_ny$, therefore $y=N\cdot n+x_0$ for some $N\in\mathbb{N}$.
Since the greatest common divisor of the numbers $\{n/n_1,\cdots,n/n_l\}$ is $1$,
 there are   $c_1,\ldots,c_l$ such that $\sum_{i}^{}c_i(n/n_i)=1$.
Since for any $i\not=j$, $(n_i,n_j)\mid u_i-u_j$ and $[n_i,n_j] \mid [n_1,\ldots,n_l]=n$
there are $d_{i,j}$ and $e_{i,j}$ such that $u_i-u_j=d_{i,j}(n_i,n_j)$
 and $n=e_{i,j}[n_i,n_j]$. Whence,  by $(n_i,n_j)[n_i,n_j]=n_in_j$ we have
$(u_i-u_j)n/n_i=d_{i,j}\!\cdot\!(n_i,n_j)\!\cdot\!e_{i,j}\!\cdot\![n_i,n_j]/n_i=
d_{i,j}\!\cdot\!e_{i,j}\!\cdot\!n_j$.
Put $x_0=\sum_{i}u_i\!\cdot\!c_i\!\cdot\!(n/n_i)$. Then \\[-2em]
\begin{center}
\begin{tabular}{lll}
$x_0$ & $=$ & $u_jc_jn/n_j+\sum_{i\not=j}u_i\!\cdot\!c_i\!\cdot\!(n/n_i)$ \\
 & $=$ & $u_jc_jn/n_j+\sum_{i\not=j}(u_i-u_j)\!\cdot\!c_i\!\cdot\!(n/n_i)
 +\sum_{i\not=j}u_j\!\cdot\!c_i\!\cdot\!(n/n_i)$ \\
 & $=$ & $u_j\!\cdot\!\sum_ic_i\!\cdot\!(n/n_i)+
 \sum_{i\not=j}(u_i-u_j)\!\cdot\!c_i\!\cdot\!(n/n_i)$\\
 & $=$ & $u_j+\sum_{i\not=j}c_i\!\cdot\!(u_i-u_j)\!\cdot\!(n/n_i)$ \\
  & $=$ & $u_j+\sum_{i\not=j}c_i\!\cdot\!d_{i,j}\!\cdot\!e_{i,j}\!\cdot\!n_j$ \\
    & $=$ & $u_j+n_j\!\cdot\!\sum_{i\not=j}c_i\!\cdot\!d_{i,j}\!\cdot\!e_{i,j},$ \\
         \end{tabular}
\end{center}
which implies the desired conclusion $x_0\equiv_{n_j}u_j$ (for every $j=1,\ldots,l$).
\end{proof}

\begin{lemma}\label{lem-rn}
The system of relations $\{R_{n_i}(u_i\!\cdot\!x)\}_{i=1}^{l}$ {\rm
(}recall that $R_n(x)\equiv \exists y[x=y^n]${\rm )} has a solution
in $\mathbb{Q}^+$ if and only if for every $i\not=j$,
$R_{(n_i,n_j)}(u_i\!\cdot\!u_j^{-1})$ holds where $(n_i,n_j)$ is the
greatest common divisor of $n_i$ and $n_j$. Moreover if
$\bigwedge_{i\not=j}R_{(n_i,n_j)}(u_i\!\cdot\!u_j^{-1})$ holds  then
for  $n=[n_1,\ldots,n_l]$ and for some fixed
$c_1,\ldots,c_l\in\mathbb{Z}$ which satisfy
$\sum_ic_i\!\cdot\!n/n_i=1$   all of the solutions  are of the form
$w^n\!\cdot\!\prod_{i=1}^{l}(u_i)^{-c_i\cdot n/n_i}$ for arbitrary
$w\in\mathbb{Q}$.
\end{lemma}
\begin{proof}
Clearly, if $R_{n_i}(u_i\!\cdot\!x)$ and  $R_{n_j}(u_j\!\cdot\!x)$
then $R_{(n_i,n_j)}(u_i\!\cdot\!x)$ and
$R_{(n_i,n_j)}(u_j^{-1}\!\cdot\!x^{-1})$, and so
$R_{(n_i,n_j)}(u_i\!\cdot\!u_j^{-1})$. Conversely, if
$\bigwedge_{i\not=j}R_{(n_i,n_j)}(u_i\!\cdot\!u_j^{-1})$ holds and
$n=[n_1,\ldots,n_l]$ then since the greatest common divisor of
$n/n_i$'s is $1$ there are some $c_1,\ldots,c_l$ such that
$\sum_ic_i\!\cdot\!n/n_i=1$. We show that
$x_0=\prod_{i=1}^{l}(u_i)^{-c_i\cdot n/n_i}$ satisfies
$\bigwedge_iR_{n_i}(u_i\!\cdot\!x_0)$. Note that every rational
number  can be uniquely factorized into a product of some (powers
of) primes $p_1^{\beta_1}p_2^{\beta_2}\cdots p_m^{\beta_m}$ where
$\beta_1,\beta_2,\ldots,\beta_m\in\mathbb{Z}$ (for example
$24/45=2^{3}3^{-1}5^{-1}$). For  a fix prime ${\sf p}$,  assume the
exponents of ${\sf p}$ in the unique factorizations of
$u_1,\ldots,u_l$  are respectively $\alpha_1,\ldots,\alpha_l$. Then
the exponent of ${\sf p}$ in the unique factorization of $x_0$ will
be $\alpha=\sum_i-c_i\alpha_i(n/n_i)$. Also, by the assumption
$\bigwedge_{i\not=j}R_{(n_i,n_j)}(u_i\!\cdot\!u_j^{-1})$ we have
$\bigwedge_{i\not=j}\alpha_i\equiv_{(n_i,n_j)}\alpha_j$. So, by the
proof of the General Chinese Remainder Theorem~\ref{th-crt},
$\bigwedge_i\alpha\equiv_{n_i}-\alpha_i$, or
$\bigwedge_in_i\mid\alpha_i+\alpha$. This means that the exponent of
(every prime) in the unique factorization of $u_i\cdot x_0$ is a
multiple of $n_i$, whence $R_{n_i}(u_i\!\cdot\!x_0)$ holds (for
$i=1,\ldots,l$). Now assume for $y\in\mathbb{Q}$ the relation
$\bigwedge_iR_{n_i}(u_i\!\cdot\!y)$ holds. Then for any prime ${\sf
p}$, if the exponent of ${\sf p}$ in the unique factorization of $y$
is $\beta$, we have $\bigwedge_i\beta\equiv_{n_i}-\alpha_i$. Whence,
by the  proof of Theorem~\ref{th-crt} we have
$\beta\equiv_n\alpha$, and so
  $y=w^n\!\cdot\!x_0$ holds for some $w\in\mathbb{Q}$.
\end{proof}

\begin{lemma}\label{lem-rm}
The system of relations $\{R_{n_j}(u_j\!\cdot\!x)\}_{j=1}^{l},\{\neg
R_{m_k}(v_k\!\cdot\!x)\}_{k=1}^{\ell}$  has a solution in
$\mathbb{Q}^+$ if and only if
$\bigwedge_{i\not=j}R_{(n_i,n_j)}(u_i\!\cdot\!u_j^{-1})\wedge\bigwedge_{k:m_k\mid
n}\neg R_{m_k}(v_k\!\cdot\!x_0)$ holds where $n=[n_1,\ldots,n_l]$
and $x_0=\prod_{i=1}^{l}(u_i)^{-c_i\cdot n/n_i}$ in which $c_i$'s
satisfy $\sum_ic_i\!\cdot\!n/n_i=1$ {\rm (}by
$(n/n_1,\ldots,n/n_l)=1${\rm )}.
\end{lemma}
\begin{proof} Suppose the system
$\{R_{n_j}(u_j\!\cdot\!x)\}_{j=1}^{l},\{\neg R_{m_k}(v_k\!\cdot\!x)\}_{k=1}^{\ell}$
has a solution (for $x$) in $\mathbb{Q}^+$. Then by Lemma~\ref{lem-rn},
 $\bigwedge_{i\not=j}R_{(n_i,n_j)}(u_i\!\cdot\!u_j^{-1})$ holds, and moreover
 $x$ is of the form $w^n\!\cdot\!x_0$ for some $w\in\mathbb{Q}^+$. We show that
  $\bigwedge_{k:m_k\mid n}\neg R_{m_k}(v_k\!\cdot\!x_0)$ holds too. Suppose
   $m_k\mid n$. Then $v_k\!\cdot\!x=v_k\!\cdot\!w^n\!\cdot\!x_0$, and so by $R_{m_k}(w^n)$
    and $\neg R_{m_k}(v_k\!\cdot\!x)$ we infer that $\neg R_{m_k}(v_k\!\cdot\!x_0)$.
    Conversely,
    suppose
    \newline\centerline{$\bigwedge_{i\not=j}R_{(n_i,n_j)}(u_i\!\cdot\!u_j^{-1})
    \wedge\bigwedge_{k:m_k\mid n}\neg R_{m_k}(v_k\!\cdot\!x_0)$.}
    Then by Lemma~\ref{lem-rn} for any $w\in\mathbb{Q}^+$ the number
    $x=w^n\!\cdot\!x_0$ satisfies $\bigwedge_j{R_{n_j}(u_j\!\cdot\!x)}$.
    We choose a suitable $w$ for which $x=w^n\!\cdot\!x_0$ also satisfies
    $\bigwedge_{k}\neg R_{m_k}(v_k\!\cdot\!x)$. Choose ${\sf P}$ be a prime
    number which does not appear in the  (unique) factorization  of any of
     $u_1$ or $u_2$ or $\ldots$ or $u_l$ or $v_1$ or $\ldots$ or $v_\ell$. Now
     we show that $x={\sf P}^n\!\cdot\!x_0$ satisfies
     $\bigwedge_{k}\neg R_{m_k}(v_k\!\cdot\!x)$:

\noindent (i) If $m_k\mid n$ then $\neg R_{m_k}(v_k\!\cdot\!x_0)$
and $R_{m_k}({\sf P}^n)$; whence $\neg R_{m_k}(v_k\!\cdot\!x)$.

\noindent (ii) If $m_k\nmid n$, then $\neg R_{m_k}({\sf P}^n)$ and
 so $\neg R_{m_k}(v_k\!\cdot\!x)$, because the prime number ${\sf P}$ does not appear in
 the unique factorization of $x_0$ or $v_k$
 (if we had $R_{m_k}(v_k\!\cdot\!x)\equiv R_{m_k}(v_k\!\cdot\!x_0\!\cdot\!{\sf P}^n)$
 then we must have had $R_{m_k}({\sf P}^n)$ or $m_k\mid n$).
\end{proof}

\begin{theorem}\label{th-q1}
The theory of the structure $\langle\mathbb{Q}^+,1,\cdot,
^{-1},R_2,R_3,\ldots\rangle$ admits quantifier elimination.
\end{theorem}
\begin{proof}
The folklore technique of quantifier elimination starts
from characterizing the terms and  atomic formulas, also eliminating
 negations, implications and universal quantifiers, and then
removing the disjunctions from the scopes of existential
quantifiers, which leaves the final case to be the existential
quantifier with the conjunction of some atomic (or negated atomic)
formulas (see Theorem 31F of \cite{enderton}).
Removing this one existential quantifier implies the
ability to eliminate all the other quantifiers by induction.

Let us
summarize the first steps:
For a variable $x$ and parameter ${a}$, all $\mathcal{L}-$terms are equal to $x^k{a}^l$ for
some $k,l\in\mathbb{Z}$. Atomic $\mathcal{L}-$formulas are in the form $s=t$ or
 $R_n(u)$ for some terms $s,t,u$ and $n\geqslant 2$. Negated atomic $\mathcal{L}-$formulas
are thus $s\not=t$  and $\neg R_n(u)$.
  So, we will eliminate the quantifier of the formulas of the form
   $\exists x (\bigwedge_i \theta_i)$
  where each $\theta_i$ is in the form $(x^\alpha=s)$ or $(x^\beta\not=t)$
  or $R_n(ux^\gamma)$ or $\neg R_m(vx^\delta)$ for some
   $\alpha,\beta,\gamma,\delta\in\mathbb{N}$ and $\mathcal{L}-$terms $s,t,u,v$.
   Whence, it suffices to show that the $\mathcal{L}-$formula
    \newline\centerline{$\exists x \big[\bigwedge_h(x^{\alpha_h}=s_h)\wedge
     \bigwedge_i(x^{\beta_i}\not=t_i)
  \wedge\bigwedge_j(R_{n_j}(u_j\cdot x^{\gamma_j}))
  \wedge\bigwedge_k(\neg R_{m_k}(v_k\cdot x^{\delta_k}))\big]$}  is
   equivalent to another
  $\mathcal{L}-$formula in which $\exists x$  does not appear.
  This will finish the proof.

   Here comes the next steps of quantifier elimination.  The powers of
$x$ can be unified: let $p$  be the least common multiple of the
$\alpha_h$'s, $\beta_i$'s, $\gamma_j$'s,  and $\delta_k$'s.  From the
$\langle\mathbb{Q}^+,\mathcal{L}\rangle-$equivalences $a\!=\!b
\leftrightarrow a^q\!=\!b^q$
and $R_n(a)\leftrightarrow R_{nq}(a^q)$  we infer that the above
formula can be re-written equivalently as
{$\exists x \big[\bigwedge_h(x^{p}\!=\!s_h)\wedge \bigwedge_i(x^{p}\not=t_i)
\wedge\bigwedge_j(R_{n_j}(u_j\cdot x^{p}))\wedge\bigwedge_k(\neg R_{m_k}(v_k\cdot x^{p}))
\big]$}
for possibly   new  $s_h$'s, $t_i$'s, $u_j$'s, $n_j$'s, $v_k$'s,
 and $m_k$'s. This formula is in turn  equivalent to
 \newline\centerline{$\exists y
 \big[\bigwedge_h(y\!=\!s_h)\wedge \bigwedge_i(y\not=t_i)
\wedge\bigwedge_j(R_{n_j}(u_j\cdot y))\wedge
\bigwedge_k(\neg R_{m_k}(v_k\cdot y))\wedge R_{p}(y)
\big]$}
 (with the substitution $y=x^p$). Thus it suffices to show that
 the following  formula is equivalent to a quantifier--free formula:
{$\exists x \big[\bigwedge_h(x\!=\!s_h)\wedge \bigwedge_i(x\not=t_i)
\wedge\bigwedge_j(R_{n_j}(u_j\cdot x))\wedge\bigwedge_k(\neg R_{m_k}(v_k\cdot x))
\big]$.}
   If the conjunction $\bigwedge_h(x=s_h)$ is not empty ($h\not=0$),
   then the above formula is
 equivalent to the quantifier--free formula (for some term $s_0$):
 \newline\centerline{
 $\bigwedge_h(s_0\!=\!s_h)\wedge \bigwedge_i(s_0\not=t_i)
\wedge\bigwedge_j(R_{n_j}(u_j\cdot s_0))\wedge\bigwedge_k(\neg R_{m_k}(v_k\cdot s_0)).$}
     So, let us assume that the conjunction $\bigwedge_h(x=s_h)$ is empty ($h=0$),
     and thus we are to eliminate the quantifier of the formula
  \newline\centerline{(A) \qquad
$\exists x \big[\bigwedge_i(x\not=t_i)
\wedge\bigwedge_j(R_{n_j}(u_j\cdot x))\wedge\bigwedge_k(\neg R_{m_k}(v_k\cdot x))
\big]$.}
By Lemma~\ref{lem-rm} this formula implies the following quantifier--free formula
\newline\centerline{(B) \qquad\qquad \qquad  $\bigwedge_{i\not=j}R_{(n_i,n_j)}
(u_i\!\cdot\!u_j^{-1})\wedge\bigwedge_{k:m_k\mid n}\neg R_{m_k}(v_k\!\cdot\!x_0)$}
where $n=[n_1,\ldots,n_l]$ and $x_0=\prod_{i=1}^{l}(u_i)^{-c_i\cdot n/n_i}$ in
 which $c_i$'s satisfy $\sum_ic_i\!\cdot\!n/n_i=1$. Also the proof of
 Lemma~\ref{lem-rm} shows that if (B) holds, then there are infinitely
  many $x$'s which satisfy $\bigwedge_j(R_{n_j}(u_j\cdot x))\wedge
  \bigwedge_k(\neg R_{m_k}(v_k\cdot x))$, and so some of those $x$'s also
  satisfy $\bigwedge_i(x\not=t_i)$; whence (A) holds too. Summing up,  (A) is equivalent to the quantifier--free formula (B).
\end{proof}

The above proof can be adapted to show the quantifier--elimination
for the theory of the structures $\langle\mathbb{Q}^{\geqslant
0},0,1,\cdot, ^{-1},R_2,R_3,\ldots\rangle$ and  $\langle\mathbb{Q},0,1,-1,\cdot,
^{-1},R_2,R_3,\ldots,\mathcal{P}\rangle$, where the predicate
$\mathcal{P}$ is defined as $\mathcal{P}(x)\iff x>0$. Since the
techniques of the proofs of these theorems are very similar to  the  proofs in the next section, we do not present them.

\begin{theorem}\label{th-q2}
Theories of the following structures admit quantifier
elimination:
\newline\centerline{$\langle\mathbb{Q}^{\geqslant
0},0,1,\cdot, ^{-1},R_2,R_3,\ldots\rangle$ and
$\langle\mathbb{Q},0,1,-1,\cdot,
^{-1},R_2,R_3,\ldots,\mathcal{P}\rangle$.}
\end{theorem}

\begin{corollary}\label{cor-q3}
The theories of the structures
$\langle\mathbb{Q}^{+},\cdot\rangle$,
$\langle\mathbb{Q}^{\geqslant 0},\cdot\rangle$, and
$\langle\mathbb{Q},\cdot\rangle$ are decidable.
\end{corollary}

\subsection{Real Numbers}
The multiplicative theory of the positive  real numbers,
$\langle\mathbb{R}^+,\cdot\rangle$, can be axiomatized as
Non--Trivial Torsion--Free Divisible Abelian Groups:
\\
{$\bullet\,\forall x,y,z\,\big(x\!\cdot\!(y\!\cdot\!z)\!=\!(x\!\cdot\!y)\!\cdot\!z\big)$
 \ \ \ }
{$\bullet\,\forall x\,\big(x\!\cdot\!1\!=\!x\!=\!1\!\cdot\!x\big)$\ \ \ }
{$\bullet\,\forall x\,\big(x\!\cdot\!x^{-1}\!=\!1\!=\!x^{-1}\!\cdot\!x\big)$\ \ \ }
{$\bullet\,\forall x,y\,\big(x\!\cdot\!y=y\!\cdot\!x\big)$}\\
{$\bullet\,\forall x\exists y \big(y^n\!=\!x\big)$, \ $n=2,3,\cdots$}
{$\bullet\,\forall x\,\big(x^n\!=\!1\rightarrow x\!=\!1\big)$, \ $n=2,3,\cdots$\qquad\qquad\ \; }
{$\bullet\,\exists x\,\big(x\neq 1\big)$}
\begin{theorem}\label{th-r1}
The theory of the structure
$\langle\mathbb{R}^+,1,\cdot,^{-1}\rangle$  admits quantifier
elimination.
\end{theorem}
\begin{proof}
Every atomic formula containing the variable $x$ can be written as
$x^n=t$ for some $n\in\mathbb{Z}$ and some term $t$. So, it suffices
to show that the formula
\newline\centerline{$\exists x\,\big(\bigwedge_ix^{n_i}=t_i\wedge\bigwedge_jx^{m_j}\not=s_j
\big)$}
is equivalent to a quantifier-free formula. Take $p$ be the lease
common multiple of $n_i$'s and $m_j$'s; then by
$a=b\longleftrightarrow a^k=b^k$, the above formula will be
equivalent to
\newline\centerline{$\exists x\,
\big(\bigwedge_ix^{p}=t_i\wedge\bigwedge_jx^{p}\not=s_j\big)$}
and by the divisibility of
$\langle\mathbb{R}^+,1,\cdot,^{-1}\rangle$ (the axiom $\forall
x\exists y \big(y^n=x\big)$) we can write this as
\newline\centerline{$\exists x\,\big(\bigwedge_ix=t_i\wedge\bigwedge_jx\not=s_j\big)$.}
If the conjunct $\bigwedge_ix=t_i$ is not empty, then the above
formula is equivalent to the following quantifier--free formula (for
some term $t_0$):
\newline\centerline{$\bigwedge_it_0=t_i\wedge\bigwedge_jt_0\not=s_j$,}
and if the conjunct $\bigwedge_ix=t_i$ is empty, then the above
formula (being $\exists x\,\big(\bigwedge_jx\not=s_j\big)$ which is
true) is equivalent to $0=0$ which is a quantifier-free formula.
\end{proof}
\begin{theorem}\label{th-r2}
The theory of the structure  $\langle\mathbb{R}^{\geqslant
0},0,1,\cdot,^{-1}\rangle$ admits quantifier elimination.
\end{theorem}
\begin{proof} This can be proved just like  Theorem~\ref{th-r1},
noting that $\forall x\, (x\!\cdot\!0\!=\!0)$.
\end{proof}

Let $\mathcal{P}$ be the predicate of positiveness:
$\mathcal{P}(x)\equiv x\!>\!0$. Below we will show that the theory
of the structure
$\langle\mathbb{R},0,1,-1,\cdot,^{-1},\mathcal{P}\rangle$ admits
quantifier elimination. By convention $x^{-1}=1/x$ if $x\not=0$ and
$x^{-1}=0$ if $x=0$. The notation $-x$ is a shorthand for
$(-1)\!\cdot\!x$.
\begin{theorem}\label{th-r3}
The theory of the structure
$\langle\mathbb{R},0,1,-1,\cdot,^{-1},\mathcal{P}\rangle$ admits
quantifier elimination.
\end{theorem}
\begin{proof}
All the atomic formulas of the language
$\{0,1,-1,\cdot,^{-1},\mathcal{P}\}$  containing the variable $x$
are in the form $x^n=t$ or $x^m\not=s$ or $\mathcal{P}(x^nt)$ or
$\neg\mathcal{P}(x^ms)$ for some $n,m\in\mathbb{Z}$ and terms $t,s$.
We note that negation can be eliminated from $\mathcal{P}$ by
$\neg\mathcal{P}(x)\longleftrightarrow x\!=\!0\vee\mathcal{P}(-x)$.
Also by
$\mathcal{P}(a\!\cdot\!b)\longleftrightarrow[\mathcal{P}(a)\wedge\mathcal{P}(b)]
\vee[\mathcal{P}(-a)\wedge\mathcal{P}(-b)]$ and $\mathcal{P}(x^2)$
we can write $\mathcal{P}(x^nu)$ as $\mathcal{P}(u)$ if $n$ is even,
and as $[\mathcal{P}(x)\wedge\mathcal{P}(u)]
\vee[\mathcal{P}(-x)\wedge\mathcal{P}(-u)]$ if $n$ is odd. Thus for
eliminating the quantifier of $\exists
x\,\big(\bigwedge_ix^{n_i}=t_i\wedge\bigwedge_jx^{m_j}\not=s_j\wedge
\bigwedge_k\mathcal{P}(x^{\alpha_k}\!\cdot\!u_k)
\wedge\bigwedge_l\neg\mathcal{P}(x^{\beta_l}\!\cdot\!v_l)\big)$ we
need to consider the formulas of the form
{$\exists x\,\big(\bigwedge_ix^{n_i}=t_i\wedge
\bigwedge_jx^{m_j}\not=s_j\wedge\mathcal{P}(x)\big)$ or}
{$\exists x\,\big(\bigwedge_ix^{n_i}=t_i\wedge
\bigwedge_jx^{m_j}\not=s_j\wedge\mathcal{P}(-x)\big)$}
only. Note that the second formula is equivalent to the following formula (by $y=-x$):
\newline\centerline{$\exists y\,\big(\bigwedge_iy^{n_i}=(-1)^{n_i}\!\cdot\!t_i\wedge
\bigwedge_jy^{m_j}\not=(-1)^{m_j}s_j\wedge\mathcal{P}(y)\big)$.}
So, we consider the following formulas for eliminating their quantifiers:
\newline\centerline{(A) \qquad $\exists x\,\big(\mathcal{P}(x)\wedge
\bigwedge_ix^{n_i}=t_i\wedge \bigwedge_jx^{m_j}\not=s_j\big)$.} Note
that when $\mathcal{P}(x)$ does not appear,  then we can again write
every formula $\exists x\,\varphi(x)$ as
\newline\centerline{$\exists x\,(\mathcal{P}(x)\wedge\varphi(x))\vee
\varphi(0)\vee\exists x\,(\mathcal{P}(x)\wedge\varphi(-x))$.}
For eliminating the quantifier of the formula (A)  we distinguish
the signs of the other variables appearing in it by the above trick.
So, if $y_1,\ldots,y_\ell$ are all the variables (other than $x$)
appearing in (A), then (A) is equivalent to the following formula
 \newline\centerline{(A') \qquad $\bigvee_{\pi\in 3^\ell}\exists x\,
 \big(\mathcal{P}(x)\wedge\bigwedge_{k=1}^{\ell}\mathcal{P}^{\pi(y_k)}(y_k)\wedge
 \bigwedge_ix^{n_i}=t_i\wedge
\bigwedge_jx^{m_j}\not=s_j\big)$} where $3^\ell=\{\pi\mid
\pi:\{y_1,\ldots,y_\ell\}\rightarrow\{-1,0,1\}\}$   and
$\mathcal{P}^{\pi(y)}(y)$ is defined as
$\mathcal{P}^{\pi(y)}(y)=\mathcal{P}(y)$ if $\pi(y)=1$,
$\mathcal{P}^{\pi(y)}(y)=``y=0"$ if $\pi(y)=0$, and
$\mathcal{P}^{\pi(y)}(y)=\mathcal{P}(-y)$ if $\pi(y)=-1$. By
changing $y_i$ to $-y_i$ when necessary, then we need to eliminate
the quantifier of the formula
\newline\centerline{(B) \qquad $\exists x\,\big(\mathcal{P}(x)\wedge
\bigwedge_{k=1}^\ell\mathcal{P}(y_k)\wedge\bigwedge_ix^{n_i}=t_i(y_1,\ldots,y_\ell)\wedge
\bigwedge_jx^{m_j}\not=s_j(y_1,\ldots,y_\ell)\big)$} where $x$ and
$y_k$'s are all the appearing variables. By the equivalences
$\mathcal{P}(z)\!\wedge\!\mathcal{P}(-z)\leftrightarrow 0=1$ and
$\mathcal{P}(u)\!\wedge\!\mathcal{P}(v)\rightarrow u\not=-v$, we can
assume that no fomula of the form $x^{n_i}=-y_1^{\alpha_1}\cdots
y_\ell^{\alpha_\ell}$ (which is false) or of the form
$x^{m_j}\not=-y_1^{\alpha_1}\cdots y_\ell^{\alpha_\ell}$ (which is
true) appears in (B). Or in the other words, no $-1$ should appear
in (B).  Finally, Theorem~\ref{th-r1} implies that the formula (B)
is equivalent to a quantifier--free formula, since $x$ and all the
$y_k$'s are supposed to be in $\mathbb{R}^+$ by the assumption
$\mathcal{P}(x)\wedge\bigwedge_{k=1}^\ell\mathcal{P}(y_k)$.
\end{proof}

Let us note that we did not use ordering    of reals in the above
proof, and in fact $<$ cannot be defined by multiplication in
$\langle\mathbb{R},\cdot\rangle$.
\subsection{Complex Numbers}
We will give a direct  proof for the decidability
of the theory of the structure $\langle\mathbb{C},\cdot\rangle$
(without referring to Tarski's theorem about the decidability of the
theory of algebraically closed fields) by showing the quantifier
elimination of a suitable extension of the language $\{\cdot\}$. For
any $n\geqslant 3$, let $\omega_n=\cos(2\pi/n)+i\sin(2\pi/n)$. The
complex number $\omega_n$ is an $n-$th root of the unit, and indeed
all the $n-$th roots of the unit are
$\omega_n,(\omega_n)^2,\cdots,(\omega_n)^n=1$. That is  to say that
$\{z\in\mathbb{C}\mid z^n=1\}=\{(\omega_n)^i\mid 1\leqslant
i\leqslant n\}$. Let us recall that by convention $0^{-1}=0$.
\begin{theorem}\label{th-c1}
The theory of
$\langle\mathbb{C},0,1,-1,\cdot,^{-1},\omega_3,\omega_4,\ldots\rangle$
admits quantifier elimination.
\end{theorem}
\begin{proof}
Just like the proof of Theorem~\ref{th-r1} it suffices to show that
the formula
\newline\centerline{(A)\qquad $\exists x\,\big(
\bigwedge_ix^{n_i}=t_i\wedge\bigwedge_jx^{m_j}\not=s_j\big)$,}
where $t_i$'s and $s_j$'s are multiplications of some variables or
the constants $0,1,-1,\omega_3,\omega_4,\ldots$ or their inverses,
is equivalent to a quantifier-free formula. We can assume that all
$n_i$'s are equal, since if for example $n_i<n_j$ then the formula
{$\exists x\,(x^{n_i}=t_i\wedge x^{n_j}=t_j\wedge\theta)$} is
equivalent to {$\exists x\,(x^{n_i}=t_i\wedge
x^{n_j-n_i}=t_j\!\cdot\!t_i^{-1}\wedge\theta)$.} Continuing this
procedure which reduces the powers of bigger exponents, we will
reach to a formula like $\exists
x\,\big(\bigwedge_ix^{n}=t_i\wedge\bigwedge_jx^{m_j}\not=s_j\big)$
which is equivalent to $\bigwedge_i t_i=t_0\wedge \exists
x\,\big(x^n=t_0\wedge\bigwedge_jx^{m_j}\not=s_j\big)$. Whence we are
to eliminate the quantifier of the formula
\qquad\qquad\qquad  {(B)\qquad $\exists
x\,\big(x^n=t\wedge\bigwedge_{j=1}^{j=l}x^{m_j}\not=s_j\big)$.}
\\
By
the equivalence
$y^k=a^k\longleftrightarrow\bigvee_{i<n}y=b(\omega_n)^i$ which holds
in $\mathbb{C}$ we have the following equivalence in $\mathbb{C}$:
$x^k\neq b \longleftrightarrow x^{kn}\neq b^n\vee
\bigvee_{0<i<n}x^k=b(\omega_n)^i$. For every $j=1,\ldots,l$ and
$k=1,\ldots,n$ let the formula $\alpha_{j,k}$ be
$x^{m_j}=s_j(\omega_n)^k$ when $k<n$ and for $k=n$ let
$\alpha_{j,k}=\alpha_{j,n}$ be the formula $x^{nm_j}\neq (s_j)^n$.
Then the formula $\bigwedge_{j=1}^{j=l}x^{m_j}\not=s_j$ is
equivalent to $\bigwedge_{j=1}^{j=l}\bigvee_{k=1}^{k=n}\alpha_{j,k}$
which is (by Propositional Logic) equivalent to
$\bigvee_{f:l\rightarrow n}\bigwedge_{j=1}^{j=l}\alpha_{j,f(j)}$
where $f:l\rightarrow n$ denotes a function from $\{1,\cdots,l\}$ to
$\{1,\cdots,n\}$. So the formula (B) is equivalent to {$\exists
x\,\big(x^n=t\wedge\bigvee_{f:l\rightarrow
n}\bigwedge_{j=1}^{j=l}\alpha_{j,f(j)}\big)$ or
$\bigvee_{f:l\rightarrow n} \exists x\,\big(
x^n=t\wedge\bigwedge_{j=1}^{j=l}\alpha_{j,f(j)} \big)$.} Hence, we
need to eliminate the quantifier of the formula $\exists
x\,\big(x^n=t\wedge\bigwedge_{j:f(j)=n}(x^n)^{m_j}\neq
(s_j)^n\wedge\bigwedge_{j:f(j)\neq
n}x^{m_j}=s_j(\omega_n)^{f(j)}\big)$. This formula is in turn
equivalent to
\newline\centerline{
$\bigwedge_{j:f(j)=n}t^{m_j}\neq (s_j)^n\wedge\exists
x\,\big(x^n=t\wedge\bigwedge_{j:f(j)\neq
n}x^{m_j}=s_j(\omega_n)^{f(j)}\big).$} So, it suffices to eliminate
the quantifier of the formulas in the form
\newline\centerline{(C)\qquad $\exists x\,\big(\bigwedge_ix^{k_i}=u_i\big)$.}
By the very same procedure that we got to the formula (B) from the
formula (A) we can see that the formula (C) is equivalent to a
formula of the form $\exists x\,(x^k=u)$, and finally this is
equivalent to the quantifier-free formula $0=0$ in
$\langle\mathbb{C},\cdot\rangle$.
\end{proof}
\section{Conclusions}
By a theorem of Tarski the theory of the structure
$\langle\mathbb{C},+,\cdot\rangle$ is decidable, whence so is the
theory of $\langle\mathbb{C},\cdot\rangle$. We showed a direct proof
of this fact in Theorem~\ref{th-c1}. Also Tarski showed the
decidability of the theory of the structure
$\langle\mathbb{R},+,\cdot\rangle$, and so the theory of the
structure $\langle\mathbb{R},\cdot\rangle$ is decidable too. We also
presented a direct proof of this fact in Theorem~\ref{th-r3}. The
references \cite{kk,bcr,bpc}  contain some beautiful proofs of
the above theorems of Tarski.

By G\"odel's incompleteness theorem the theory of  the structure
$\langle\mathbb{N},+,\cdot\rangle$ is not decidable (see e.g.
\cite{enderton} or \cite{smorynski} for a proof), and by
 the four square theorem of Lagrange the set $\mathbb{N}$
 is definable in $\langle\mathbb{Z},+,\cdot\rangle$:
{for every $n\in\mathbb{Z}$, $n\in\mathbb{N}\iff \exists\,\alpha,\beta,\gamma,\delta\,
(n=\alpha^2+\beta^2+\gamma^2+\delta^2).$}
Whence, the theory of the structure
$\langle\mathbb{Z},+,\cdot\rangle$ is not decidable as well. Also,
the theory of the structure $\langle\mathbb{Q},+,\cdot\rangle$ is
not decidable, since $\mathbb{Z}$ is definable in it (see
\cite{robinson}). However, the multiplicative theory
of natural and integer numbers, the theories of
$\langle\mathbb{N},\cdot\rangle$ and
$\langle\mathbb{Z},\cdot\rangle$, are decidable (see
\cite{cegielski,smorynski}). The decidability or
undecidability of the theory of $\langle\mathbb{Q},\cdot\rangle$ had
remained open in the literature. In Theorem~\ref{th-q1} we showed
the decidability of the theory of
$\langle\mathbb{Q}^+,\cdot\rangle$, and by the technique of the
proof of Theorem~\ref{th-r3} one can show the decidability of the
theory of $\langle\mathbb{Q},\cdot\rangle$ (Corollary~\ref{cor-q3}).


\end{document}